\documentclass[12pt]{amsart}
\usepackage{fullpage}
\usepackage{times}
\usepackage{latexsym}
\usepackage{amssymb}
\usepackage[all]{xy}

\newtheorem{thm}{Theorem}[section]

\newtheorem{lem}[thm]{Lemma}

\theoremstyle{remark}
\newtheorem{rem}[thm]{Remark}

\usepackage{hyperref}

\newcommand{\CP}{\mathbb{CP}}

\newcommand{\Q}{\mathbb{Q}}

\newcommand{\Z}{\mathbb{Z}}
\newcommand{\C}{\mathbb{C}}

\newcommand{\SO}{\mathrm{\SO}}

\title[Sets of degrees of maps between $SU(2)$-bundles over $S^5$]
{Sets of degrees of maps between SU(2)-bundles over the 5-sphere}

\author{Jean-Fran\c cois Lafont and Christoforos Neofytidis}
\address{Department of Mathematics, Ohio State University, Columbus, OH 43210, USA}
\email{jlafont@math.ohio-state.edu}
\address{Section de Math\'ematiques, Universit\'e de Gen\`eve, 2-4 rue du Li\`evre, Case postale 64, 1211 Gen\`eve 4, Switzerland}
\email{Christoforos.Neofytidis@unige.ch}
\date{June 5, 2018 (final version accepted for publication ) and \today \ (erratum)}

\begin{document}

\begin{abstract}
We compute the sets of degrees of maps between principal $SU(2)$-bundles over $S^5$, i.e. between any of the manifolds $SU(2)\times S^5$ and $SU(3)$. We show that the Steenrod squares provide the only obstruction to the existence of a mapping degree between these manifolds, and construct explicit maps realizing each integer that occurs as a mapping degree.

\medskip

\noindent{\em Added Erratum.} After this manuscript was accepted for publication by {\em Transformation Groups}, Xueqi Wang~\cite{Wang} pointed out a mistake in our paper. At the end of this arXiv version we add an erratum, where we correct the statement and the proof of Theorem \ref{t:s3bundlesovers5}.
\end{abstract}

\maketitle

\section{introduction}

A fundamental question in topology is whether, given two closed oriented $n$-dimensional manifolds $M$ and $N$, there is a map $f\colon M\longrightarrow N$ of degree $\deg(f)\neq0$. Recall that a continuous map $f\colon M\longrightarrow N$ has degree $d$ if $f_*([M]) = d\cdot [N]$, where $f_*\colon H_n(M)\longrightarrow H_n(N)$ is the induced homomorphism in homology and $[M]\in H_n(M)$ is the fundamental class of $M$. The dual formulation in cohomology says that $\deg(f)=d$ if $f^*(\omega_N)=d\cdot\omega_M$, where $f^*\colon H^n(N)\longrightarrow H^n(M)$ is the induced homomorphism in cohomology and $\omega_M\in H^n(M)$ is the cohomological (dual) fundamental class of $M$.

The set of degrees of maps from $M$ to $N$ is defined as
\[
D(M,N)=\{d\in\Z \ | \ \exists \ f\colon M\longrightarrow N, \ \deg(f)=d\}.
\]
For $M=N$, we write $D(M)$ to denote the set of degrees of self-maps of $M$.

In general, it is a difficult question to determine whether a given integer can be realized as a mapping degree between two manifolds. The answer is well-known in dimensions one and two. A fairly complete answer is known for self-mapping degrees in dimension three~\cite{SWWZ}, for certain classes of product manifolds~\cite{Neo}, and for maps between certain highly connected manifolds~\cite{DW}. Obstructions to the existence of a map of non-zero degree or of a particular mapping degree have been developed using a variety of tools of algebraic topology.  
One of the most classical methods is to compare the cohomology rings of $M$ and $N$. However, when 
\[
H^*(M)\cong H^*(N),
\]
then major obstructions such as the ranks of (co)homology groups, the injectivity of induced homomorphisms in cohomology or the (sub)ring structures themselves no longer suffice to answer this question.

Our goal is to investigate manifolds $M$ and $N$ with isomorphic cohomology rings and find possible other obstructions to the existence of a mapping degree in $D(M,N)$ or $D(N,M)$. In this paper, we consider the two principal $SU(2)$-bundles over $S^5$, namely $SU(2)\times S^5\cong S^3\times S^5$ and $SU(3)$, and prove the following:

\begin{thm}\label{t:s3bundlesovers5}
The sets of degrees of maps between principal $SU(2)$-bundles over $S^5$ are given as follows:
\begin{itemize}
\item[(a)] $D(S^3\times S^5)=\Z$;
\item[(b)] $D(S^3\times S^5, SU(3))=2\Z$;
\item[(c)] $D(SU(3),S^3\times S^5)=2\Z$;
\item[(d)] $D(SU(3))=4\Z\cup\{2k+1:k\in\Z\}$.
\end{itemize}
\end{thm} 

The bundles $S^3\times S^5$ and $SU(3)$ indeed have isomorphic cohomology rings. However, their Steenrod squares behave differently on degree three cohomology, since $Sq^2$ is trivial for the product $S^3\times S^5$, but an isomorphism for the non-trivial bundle $SU(3)$, mapping the generator $\overline{\alpha}\in H^3(SU(3);\Z_2)$ to the generator $\overline{\beta}\in H^5(SU(3);\Z_2)$; cf. Section \ref{bundlesintro}. As we will see in the course of the proof of Theorem \ref{t:s3bundlesovers5}, the non-triviality of the Steenrod square on $SU(3)$ implies that odd numbers cannot be realized as degrees of maps between $SU(3)$ and $S^3\times S^5$ (in those cases, $Sq^2(\overline{\alpha})$ appears once in the computation) and numbers of type $2\cdot(\text{odd})$ cannot be realized as degrees of self-maps of $SU(3)$ (in that case, $Sq^2(\overline{\alpha})$ appears twice in the computation). The complete computation of Theorem \ref{t:s3bundlesovers5} shows that this is the only obstruction to the existence of a mapping degree for maps between $SU(2)$-bundles over $S^5$. Moreover, our computation of $D(SU(3))$ corrects a mistake in a previous computation~\cite{Pue}. 

\subsection*{Outline} 

In Section \ref{ex.known} we discuss some known obstructions to the existence of maps of non-zero degree for manifolds with non-isomorphic cohomology rings. In Section \ref{bundlesintro} we briefly overview the $SU(2)$-bundles over $S^5$ and in Section \ref{proof} we prove Theorem \ref{t:s3bundlesovers5}.

\subsection*{Acknowledgments}
C. Neofytidis is grateful to Shicheng Wang for useful discussions. He also gratefully acknowledges the hospitality of Peking University where part of this research was carried out. J.-F. Lafont was partially supported by the NSF, under grant DMS-1510640.

\section{Manifolds with non-isomorphic cohomology rings}\label{ex.known}

In this section we describe a few well-known examples of manifolds with non-isomorphic cohomology rings and explain in each case the obstruction to the existence of map of non-zero degree. We refer to~\cite{dH} for a survey on this type of examples.

\subsection{Different Betti numbers}
If $f\colon M\longrightarrow N$ is a map of non-zero degree, then by Poincar\'e duality we obtain that the induced homomorphisms $f_*\colon H_*(M;\Q)\longrightarrow H_*(N;\Q)$ are surjective. In particular, the Betti numbers of $M$ are greater than or equal to the Betti numbers of $N$. Thus, for example, there is no map of non-zero degree from $S^{2n}$ to $\CP^n$ for all $n>1$.

\subsection{Same Betti numbers but different cohomology generator degrees}
Equivalent to the surjectivity in homology, the induced homomorphisms $f^*\colon H^*(N;\Q)\longrightarrow H^*(M;\Q)$ are injective whenever $f\colon M\longrightarrow N$ has non-zero degree. Thus there is no map of non-zero degree between $S^2\times S^4$ and $\CP^3$, although these manifolds have isomorphic (co)homology groups and thus equal Betti numbers. Indeed the cohomology {\it rings} of $S^2\times S^4$ and $\CP^3$ are given by 
$H^*(S^2\times S^4)=\Lambda[\alpha,\beta]$ and $H^*(\CP^3)=\Lambda[\gamma]$, where $\alpha$ and $\beta$ have degree two and $\gamma$ has degree four, which means that neither of those rings injects into the other.

\subsection{Same cohomology generator degrees but different (sub)rings}
The cohomology rings of the manifolds $\CP^2\#\CP^2$ and $\CP^2\#\overline{\CP^2}$ are generated by two elements, both of degree two. More precisely, 
$H^*(\CP^2\#\CP^2)=\Lambda[\alpha,\beta]$,
where $\alpha$ and $\beta$ have degree two with $\alpha^2=\beta^2$, and
$ H^*(\CP^2\#\overline{\CP^2})=\Lambda[\alpha,\beta]$,
where $\alpha$ and $\beta$ have degree two with $\alpha^2=-\beta^2$. In particular, neither of the those rings is isomorphic to a subring of the other, which implies that there are no maps of non-zero degree between $\CP^2\#\CP^2$ and $\CP^2\#\overline{\CP^2}$. 

Note that the information encoded in the above cohomology rings reflect the intersection forms of $\CP^2\#\CP^2$ and $\CP^2\#\overline{\CP^2}$; cf.~\cite{DW}.

\begin{rem}
The examples above all focus on differences in the cohomology, either at the level of groups, or at the level of the ring structure. 
Our arguments continue this trend, by focusing on the differences in the structure of the $\Z_2$-cohomology, viewed as a module over the 
Steenrod algebra. 
\end{rem}

\section{Principal $SU(2)$-bundles over $S^5$}\label{bundlesintro}

We now turn to the examples studied in this paper, of manifolds with isomorphic cohomology rings. Recall that principal $SU(2)$-bundles over $S^5$ are classified by $\pi_4(SU(2))\cong\Z_2$, and so there exist two such bundles: the trivial bundle $SU(2)\times S^3\cong S^3\times S^5$, and the twisted bundle $SU(3)$. 
For the latter bundle, recall that 
\[
S^3\cong SU(2)=\biggl\{\left( \begin{array}{cc}
  a & b \\
  -\overline{b} & \overline{a} \\
\end{array} \right)  : \ a,b\in\C, \ |a|^2+|b|^2=1 \biggl\}
\]
can be embedded in $SU(3)$ via
\[
SU(2)\hookrightarrow SU(3) :
\left( \begin{array}{cc}
  a & b \\
  -\overline{b} & \overline{a} \\
\end{array} \right) \mapsto
\left( \begin{array}{ccc}
  a & b & 0\\
  -\overline{b} & \overline{a} & 0 \\
  0 & 0 & 1\\
\end{array} \right).
\]
The Lie group $SU(3)$ acts on the unit sphere $S^5 \subset \mathbb C^3$, and the
stabilizer of the vector $(0,0,1)\in S^5$ is precisely the embedded $SU(2)$ described above. Thus the quotient $SU(3)/SU(2)$ is 
homeomorphic to $S^5$ and a homeomorphism is given by the orbit map
\[
SU(3)/SU(2)\longrightarrow S^5 : [A]\mapsto A\cdot \left( \begin{array}{c}
  0\\
  0\\
  1\\
\end{array} \right), \ A\in SU(3). 
\]
Concretely, this means that the projection map of the bundle $SU(3)\stackrel{p}\longrightarrow S^5$ is the projection to the third column of $A$.

Both $S^3\times S^5$ and $SU(3)$ have cohomology rings isomorphic to $\Lambda[\alpha,\beta]$, with generators $\alpha$ and $\beta$ in cohomology of degree three and five respectively. However, $SU(3)$ is not homotopy equivalent to $S^3\times S^5$, because the Steenrod square is (obviously) trivial on $H^3(S^3\times S^5;\Z_2)$, whereas the Steenrod square on $SU(3)$ is an isomorphism from $H^3(SU(3);\Z_2)$ to $H^5(SU(3);\Z_2)$; cf.~\cite{BS}. Of course, $S^3\times S^5$ and $SU(3)$ can be distinguished by their homotopy groups as well~\cite{MiT}.

\section{Proof of Theorem \ref{t:s3bundlesovers5}}\label{proof} 

In this section we prove Theorem \ref{t:s3bundlesovers5}. 

\medskip

Before beginning our proof, we discuss briefly the degree of a bundle map. 
Given two fiber bundles $F\hookrightarrow E_i \stackrel{\pi_i}\longrightarrow B_i$ with orientation preserving holonomy, where $F$ and $B_i$ are connected manifolds, let $\phi\colon E_1\longrightarrow E_2$ be a bundle map, i.e. there is a map $\phi_B\colon B_1\longrightarrow B_2$ on base spaces such that $\phi_B\circ \pi_1=\pi_2\circ\phi$. From the fact that all the spaces are connected, there is also an induced map $\phi_F$ from the fiber of $E_1$ to the fiber of $E_2$, which is well-defined up to homotopy. In particular, one can calculate the integers $\deg(\phi_B), \deg(\phi_F)$. In this setting, we have the following presumably well-known result:

\begin{lem}\label{l:degreebundlemap}
$\deg(\phi) = \deg(\phi_F) \cdot \deg(\phi_B)$.
\end{lem}

\begin{proof}
We will calculate $\deg(\phi)$ by taking a specific homotopy of $\phi$, and count oriented pre-images of a point $x\in E_2$ (we use the definition of degree from differential topology). The point 
$x$ projects to a point $q \in B_2$ in the base. One can then use the local product structure to view $x = (p,q)$ where $p\in F_q$, the fiber
above $q$. 
Homotope $\phi$, through bundle maps, so that the base map $\phi_B$ is transverse to $q$. Then further homotope $\phi$ via fiber-preserving maps
so that, for each point $q_i\in B_1$ satisfying $\phi_B(q_i) =q$, the restriction of the resulting map to the fiber $F_{q_i} \subset E_2$ is transverse to 
the point $p\in F_q$. By abuse of notation, we still call the resulting map $\phi$.

Now the pre-image of $x$ will consist of a finite collection of points, each of the form $(p_{j}, q_i)$ (in suitable local product
structure) where the various $p_j\in F_{q_i}$ for $j\in I_i$ (each indexing set $I_i$ depends on the corresponding $i$). 
One can take the orientation on the total spaces to be locally given by the wedge
of the fiber orientation with the base orientation -- this is well-defined since the holonomy is orientation preserving. Picking a horizontal lift
of $T_{q_i}B_1$ at each point $(p_j, q_i)$, we can decompose $T_{(p_j, q_i)}E_1 = T_{p_j}F_{q_i} \oplus \overline{T_{q_i}B_1}$. Similarly,
we decompose $T_{(p,q)}E_2= T_pF_q \oplus \overline{T_qB_2}$.

To compute the degree, we need to look at the sign of the determinant of $d\phi_{(p_j, q_i)}$ at each of these pre-image points. Denote
by $\phi_i$ the restriction of $\phi$ to the fiber $\phi_i\colon F_{q_i} \longrightarrow F_q$ (and note that $\phi_i \simeq \phi_F$). From our
choice of bases, we see that the matrix for $d\phi_{(p_j, q_i)}$ takes the block form $\left( \begin{array}{cc}
d(\phi_i)_{p_j} & * \\
 0 & d(\phi_B)_{q_i} \\
\end{array} \right)$, and hence $\det(d\phi_{(p_j, q_i)}) = \det(d(\phi_i)_{p_j})\cdot \det(d(\phi_B)_{q_i})$. For a matrix $A$, we will denote by
$o(A)$ the sign of $\det(A)$. Then it follows that $o(d\phi_{(p_j, q_i)})=o(d(\phi_i)_{p_j})\cdot o(d(\phi_B)_{q_i})$. We now have from the 
definition of degree that
\begin{align*}
\deg(\phi) &= \sum_{(p_j, q_i)} o(d\phi_{(p_j, q_i)}) = \sum_{(p_j, q_i)} \big[ o(d(\phi_i)_{p_j})\cdot o(d(\phi_B)_{q_i}) \big]\\
& = \sum _i \Big[\Big(\sum _{j\in I_i} o(d(\phi_i)_{p_j})\Big)\cdot o(d(\phi_B)_{q_i}) \Big]= \sum _i \big[\deg(\phi_i) \cdot o(d(\phi_B)_{q_i})\big]\\
& = \deg(\phi_F) \sum _i o(d(\phi_B)_{q_i}) = \deg(\phi_F) \cdot \deg(\phi_B),
\end{align*}
where we use the fact that all the maps $\phi_i$ are homotopic to the fiber map $\phi_F$, hence have the same degree. This concludes
the proof. 
\end{proof}

Now, we are ready to prove Theorem \ref{t:s3bundlesovers5}. Item (a) is trivial because $D(S^n)=\Z$, and so we deal with the rest of the computations. As mentioned in the introduction, for a closed oriented $n$-dimensional manifold $M$, we denote by $\omega_M\in H^n(M; \Z)$ the generator determined by the orientation of $M$.

\begin{lem}\label{l1}
$D(S^3\times S^5, SU(3))=2\Z$.
\end{lem}
\begin{proof}
Let $f\colon S^5\longrightarrow S^5$ be a self-map of even degree. We pull-back the bundle $SU(3)\stackrel{p}\longrightarrow S^5$ along $f$ to obtain an $SU(2)$-bundle $f^*(SU(3))$ over $S^5$ together with a map $\tilde{f}\colon f^*(SU(3))\longrightarrow SU(3)$, which has degree $\deg(f)$ by Lemma \ref{l:degreebundlemap}. Since $SU(2)$-bundles over $S^5$ are classified by $\pi_4(SU(2))=\Z_2$, and $\deg(f)$ is even, we deduce that $f^*(SU(3))=S^3\times S^5$. Thus $2\Z\subseteq D(S^3\times S^5, SU(3))$.

Conversely, suppose $f\colon S^3\times S^5\longrightarrow SU(3)$ is an arbitrary map of degree $\deg(f)$.  Let $\alpha$ and $\beta$ be generators of $H^3(SU(3);\Z)$ and $H^5(SU(3);\Z)$ respectively such that $\alpha\cup\beta=\omega_{SU(3)}$. Then from the definition of degree, we have that $f^*(\alpha)\cup f^*(\beta) = \deg(f)\cdot \omega_{S^3\times S^5}$. 
Each of the elements $f^*(\alpha), f^*(\beta)$ are multiples of the generators $\omega_{S^3}\times 1$ and $1\times \omega_{S^5}$ respectively. We will show
that $f^*(\beta)$ must be an {\it even multiple} of $1\times \omega_{S^5}$, which immediately implies that $\deg(f)$ is even.

We will use bars to denote the image of a cohomology class under the change of coefficient morphism $H^*(X; \mathbb \Z) \longrightarrow H^*(X; \mathbb \Z_2)$. Thus, we have that $\overline{\alpha}, \overline{\beta}$ are the generators of $H^3(SU(3);\Z _2)$ and $H^5(SU(3);\Z _2)$ respectively. Moreover, for any element $x\in H^*(SU(3); \Z)$, we have $\overline{f^*(x)} = f^*(\overline{x})$. Since $Sq^2\colon H^3(SU(3);\Z_2)\longrightarrow H^5(SU(3);\Z_2)$ is an isomorphism, we have that $Sq^2(\overline{\alpha}) = \overline{\beta}$. Using that $Sq^2\colon H^3(S^3\times S^5;\Z_2)\longrightarrow H^5(S^3\times S^5;\Z_2)$ is the zero morphism, we have the sequence of equalities
\[
\overline{f^*(\beta)}= f^*(\overline{\beta}) = f^*Sq^2(\overline{\alpha})=Sq^2\big({f^*(\overline{\alpha})}\big)=0.
\]
Thus $f^*(\beta)$ is an even multiple of the generator $1\times \omega_{S^5}$, completing the proof.
\end{proof}

Next, we prove one of the inclusions of item (c):

\begin{lem}\label{l2}
$D(SU(3),S^3\times S^5)\subseteq2\Z$.
\end{lem}
\begin{proof}
Suppose $f\colon SU(3)\longrightarrow S^3\times S^5$ is a map of degree $\deg(f)$. As before, let $\alpha$ and $\beta$ be generators of $H^3(SU(3);\Z)$ and $H^5(SU(3);\Z)$ respectively. We have $f^*(\omega_{S^3}\times 1)=\kappa\cdot\alpha$ and $f^*(1\times \omega_{S^5})=\lambda\cdot\beta$, for some $\kappa,\lambda\in\Z$. In particular, $\deg(f)=\kappa\lambda$. Again, because the Steenrod square is an isomorphism from $H^3(SU(3);\Z_2)$ to $H^5(SU(3);\Z_2)$ and trivial on $H^3(S^3\times S^5;\Z_2)$ we obtain
\[
\kappa\cdot\overline{\beta}=\kappa\cdot Sq^2(\overline{\alpha})=Sq^2f^*(\overline{\omega_{S^3}\times 1})=f^*Sq^2(\overline{\omega_{S^3}\times 1})=0.\]
This means that $\kappa$ must be even, and so $\deg(f)$ is even.
\end{proof}

In order to prove the reverse inclusion of item (c), we need to construct some maps with controlled degree. This is established in

\begin{lem}\label{l3}
$2\Z\subseteq D(SU(3),S^3\times S^5)$.
\end{lem}

\begin{proof}
Since $D(S^3\times S^5)=\Z$, it suffices to construct a single map $h\colon SU(3)\longrightarrow S^3\times S^5$ satisfying $\deg(h)=2$. We start
by defining a self-map $g\colon SU(3)\longrightarrow SU(3)$, and will see that $g$ factors through the desired map $h$.

So let $A=\left( \begin{array}{ccc}
  a & b & u\\
  c & d & v \\
  p & q & w\\
\end{array} \right)\in SU(3)$ and define the map $g\colon SU(3)\longrightarrow SU(3)$ by the formula
\[
{\footnotesize g(A)=
       \left( \begin{array}{ccc}    
-u^2(\overline{ab}+a\overline{b})+(uv-\overline{w})(-a^2+|b|^2)  & u^2(\overline{a}^2-|b|^2)-(uv-\overline{w})(ab+\overline{a}b)  & uw+\overline{v}\\
   -(uv+\overline{w})(\overline{ab}+a\overline{b})+v^2(-a^2+|b|^2)  & (uv+\overline{w})(\overline{a}^2-|b|^2)-v^2(ab+\overline{a}b)  & vw-\overline{u}\\   
  -(uw-\overline{v})(\overline{ab}+a\overline{b})+(\overline{u}+vw)(-a^2+|b|^2)  & (uw-\overline{v})(\overline{a}^2-|b|^2)-(\overline{u}+vw)(ab+\overline{a}b)  & w^2  
\end{array} \right).}
\]

If $A=\left( \begin{array}{ccc}
 a & b & 0\\
   -\overline{b} & \overline{a} & 0 \\
     0 & 0 & 1\\
\end{array} \right)\in SU(2) \subseteq SU(3)$, then it is easy to check that
\[
g(A)=\left( \begin{array}{ccc}
  a^2-|b|^2 & ab+\overline{a}b & 0\\
   -\overline{ab}-a\overline{b} & \overline{a}^2-|b|^2 & 0 \\
  0 & 0 & 1\\
\end{array} \right) = A^2. 
\]
This tells us that $g(SU(2))\subseteq SU(2)$, and the restriction of $g$ to $SU(2)$ is the squaring
map of degree $\deg(g|_{SU(2)}) =2$. 

Next let us verify that $g$ is in fact a bundle map. For any $A, B\in SU(3)$, we have that $p(A)=p(B)$ if and only if $B=AU$ for some $U\in SU(2)\subseteq SU(3)$. In that case, $A$ and $B=AU$ have the same third columns, say 
$\left( \begin{array}{c}
 u\\
  v \\
  w\\
\end{array} \right)\in S^5$. This means that $g(A)$ and $g(AU)$ have also the same third columns, namely  $\left( \begin{array}{c}
 uw+\overline{v}\\
  vw-\overline{u} \\
  w^2\\
\end{array} \right)\in S^5$. Thus there is a well-defined induced self-map $f\colon S^5\longrightarrow S^5$ given by 
\[
 \left( \begin{array}{c}
 u\\
  v \\
  w\\
\end{array} \right)\mapsto \left( \begin{array}{c}
 uw+\overline{v}\\
  vw-\overline{u} \\
  w^2\\
\end{array} \right),
\]
which has degree $2$ (see also Theorem 2.1 of \cite{PR}) and such that $p\circ g=f\circ p$.

We pull-back the bundle $SU(3)\stackrel{p}\longrightarrow S^5$ along $f$ to obtain an $SU(2)$-bundle over $S^5$ 
\[
f^*(SU(3))=\{(S,T)\in S^5\times SU(3) : f(S)=p(T)\}.
\]
Since $\deg(f)=2$, this pull-back bundle $f^*(SU(3))$ is in fact the trivial bundle -- we denote by $\tilde f$ the bundle map
$\tilde f : f^*(SU(3)) \longrightarrow SU(3)$ (projection onto the second factor).  Since $g$ is a bundle map, it factors through the pull-back bundle $f^*(SU(3))$. Thus we have that $g=\tilde f \circ h$, where $h\colon SU(3)\longrightarrow f^*(SU(3))$ is the map given by $h(A):= (p(A),g(A))$. To complete the proof, we just need to check that $\deg(h)=2$. From the multiplicativity of the degree, it suffices to check that $\deg(\tilde f)=2$
and $\deg(g)=4$. But Lemma \ref{l:degreebundlemap} implies that $\deg(\tilde f) =\deg(\tilde f_{SU(2)})\cdot\deg(\tilde f_{S^5})=2\cdot 1 =2$ (see also Lemma \ref{l1}) 
and $\deg(g)=\deg(g_{SU(2)})\cdot\deg(g_{S^5})=\deg(g\vert_{SU(2)})\cdot\deg(f)=2\cdot 2=4$ (since $g$ covers $f$ and on fiber above the point $(0,0,1)\in S^5$ is the squaring map on $SU(2)$). This finishes the proof.
\end{proof}

Lemma \ref{l3} together with Lemma \ref{l2} complete the proof for item (c) of Theorem \ref{t:s3bundlesovers5}. Finally, we prove item (d).

\begin{lem}\label{l4}
$D(SU(3))=4\Z\cup\{2k+1:k\in\Z\}.$
\end{lem}
\begin{proof}
Items (b) and (c) give $4\Z\subseteq D(SU(3))$. Now let $\alpha$ and $\beta$ be generators of $H^3(SU(3);\Z)$ and $H^5(SU(3);\Z)$ respectively such that $\alpha\cup\beta=\omega_{SU(3)}$. For any self-map $g\colon SU(3)\longrightarrow SU(3)$, we have $\deg(g)\cdot \omega_{SU(3)}= g^*(\alpha)\cup g^*(\beta)$. If $g^*(\alpha)=n\cdot\alpha$, then since the Steenrod square $Sq^2\colon H^3(SU(3);\Z_2)\longrightarrow H^5(SU(3);\Z_2)$
is an isomorphism, 
we deduce that 
\begin{equation}\label{eq:even}
g^*(\overline{\beta})=g^*Sq^2(\overline{\alpha})=Sq^2g^*(\overline{\alpha})=n\cdot Sq^2(\overline{\alpha})=n\cdot\overline{\beta}.
\end{equation}
This means that $g^*(\alpha)$ is an even multiple of $\alpha$ if and only if $g^*(\beta)$ is an even multiple of $\beta$. We deduce that $2m\notin D(SU(3))$ for $m$ odd.

Finally, it remains to show that every odd integer is realized as a mapping degree of a self-map of $SU(3)$. Let $f_m\colon S^5\longrightarrow S^5$ be a map of any odd degree $m$. As before, we pull-back the bundle $SU(3)\stackrel{p}\longrightarrow S^5$ along $f_m$ to obtain an $SU(2)$-bundle $f_m^*(SU(3))$ over $S^5$ together with a map $\tilde{f}_m\colon f_m^*(SU(3))\longrightarrow SU(3)$ of the same degree $m$. We know by item (b) of Theorem \ref{t:s3bundlesovers5} that $D(S^3\times S^5, SU(3))=2\Z$, which means that $f_m^*(SU(3))$ cannot be the trivial bundle. Thus $\tilde{f}_m$ is a self-map of $SU(3)$.
\end{proof}

This completes the proof of Theorem \ref{t:s3bundlesovers5}.

\section{Concluding Remarks}\label{rem}

We finish our discussion with a few remarks on our main result and its proof.

\begin{rem}
Item (d) of Theorem \ref{t:s3bundlesovers5} fixes a mistake in~\cite{Pue}, where it was claimed that $D(SU(3))=\{4^m\cdot (2k+1):m,k\in\Z\}$. The proof of Lemma 5.4 in~\cite{Pue} is not correct. More precisely, the proof of that lemma shows only that, whenever the degree of a self-map of $SU(3)$ is even, then it must be divisible by $4$ (and not a power of $4$ as claimed in~\cite{Pue}). In fact, the argument given in~\cite{Pue} is identical to the one we have seen in equation (\ref{eq:even}) in the proof of Lemma \ref{l4}.
\end{rem}

\begin{rem}
One of the main results in~\cite{Pue} is the construction of self-maps $\psi_m\colon SU(3)\longrightarrow SU(3)$ of any odd degree $m$. These maps come as a consequence of the study of cohomogeneity one manifolds. In Lemma \ref{l4}, we constructed self-maps $\tilde{f}_m\colon SU(3)\longrightarrow SU(3)$ of any odd degree $m$ in a rather simple way, by pulling back the bundle $SU(3)\stackrel{p}\longrightarrow S^5$ along self-maps $f_m\colon S^5\longrightarrow S^5$ of degree $m$. For $m\neq 1$, the maps $\tilde{f}_m$ are not homotopic to $\psi_m$. To see this,
let $\alpha$ and $\beta$ be the generators of $H^3(SU(3);\Z)$ and $H^5(SU(3);\Z)$ respectively such that $\alpha\cup\beta=\omega_{SU(3)}$. The Gysin sequence for $SU(3)$ gives
\[
0=H^1(S^5)\longrightarrow H^5(S^5)\stackrel{p^*}\longrightarrow H^5(SU(3))\longrightarrow H^2(S^5)=0,
\]
that is $\beta=p^*(\omega_{S^5})$, and so we obtain
\begin{equation}\label{eq:degreef}
\tilde{f}_m^*(\beta)=\tilde{f}_m^* p^*(\omega_{S^5})=p^*f_m^*(\omega_{S^5})=\deg(f_m)\cdot p^*(\omega_{S^5})=m\cdot\beta.
\end{equation}
However, $\psi_m^*(\beta)=\beta$ by Lemma 5.5 of~\cite{Pue}.
\end{rem}

\begin{rem}
As explained in the introduction, Theorem \ref{t:s3bundlesovers5} and its proof show that, in all cases, the Steenrod squares provide the only obstruction to the existence of a map of non-zero degree between principal $SU(2)$-bundles over $S^5$. Once this obstruction does not apply for an integer $d$, then $d$ is realized as a mapping degree. More generally, it would be interesting to find other classes of manifolds where the only obstructions to mapping degrees arise from the structure of the mod $p$ cohomology rings as modules over the corresponding Steenrod algebras.
\end{rem}

\bibliographystyle{amsplain}

\section{Added Erratum}

Part (c) of Theorem \ref{t:s3bundlesovers5} of our paper 
is not correct. The argument for part (c) relied on the existence of the self-map $g$ 
of $SU(3)$ given in the proof of Lemma~\ref{l3}. However, $\mathrm{im}(g)\notin SU(3)$. In fact, such a map $g$ cannot exist, and 
part (c) should instead be $D(SU(3),S^3\times S^5)=4\mathbb{Z}$. A proof of both these statements was given by Xueqi Wang \cite{Wang}. 

%\vskip 5pt

Our proof of part (d) of Theorem \ref{t:s3bundlesovers5} relied on part (c), and hence is not correct. Nevertheless, the statement of part (d) is correct as stated. 
In this erratum, we correct the proof of part (d) (our proof is independent of the results in \cite{Wang}): 
  
\begin{proof}[Proof of part (d)]
Let $\mu: SU(3)\times SU(3)\rightarrow SU(3)$ denote the Lie group multiplication. 
Given a pair of self-maps $f,g:SU(3)\rightarrow SU(3)$, we can then define a new map $\mu(f,g):SU(3)\rightarrow SU(3)$ by setting
$$\mu(f,g) (A) = f(A)\cdot g(A)$$
Let $\alpha$ and $\beta$ be generators of $H^3(SU(3);\mathbb{Z})$ and $H^5(SU(3);\mathbb{Z})$ respectively such that $\alpha\cup\beta=\omega_{SU(3)}$, the generator
for $H^8(SU(3);\mathbb{Z})$. Then for any self-map $h: SU(3)\rightarrow SU(3)$, we have that $h^*(\alpha)=n\cdot \alpha$ and $h^*(\beta)=m\cdot \beta$ for some pair of integers
$(n,m)\in \mathbb Z\times \mathbb Z$. We call this pair of integers the {\it multi-degree} of the map $h$, denoted $\text{mdeg}(h)$. Of course, if $\text{mdeg}(h)=(n,m)$, then 
$\deg(h)=nm$, so the multi-degree provides a refinement of the notion of degree.

Notice that the multi-degree is {\it additive} with respect to multiplication of maps, i.e. $\text{mdeg}(\mu(f,g)) = \text{mdeg}(f)+\text{mdeg}(g)$. Let $\text{mdeg}(f)=(n_1, m_1)$
and $\text{mdeg}(g)=(n_2, m_2)$. If we denote by 
$\mu: SU(3)\times SU(3)\rightarrow SU(3)$ the Lie group multiplication map $(A, B) \mapsto A\cdot B$, then the map $\mu(f,g)$ factors as
\[
SU(3)\xrightarrow{(f,g)} SU(3)\times SU(3)\stackrel{\mu}\longrightarrow SU(3).
\]
Thus we have
\begin{align*}
\mu(f,g)^*(\alpha) &= (f,g)^*\mu^*(\alpha) = (f,g)^*(\alpha\otimes 1 + 1\otimes\alpha)\\
&=f^*(\alpha) +g^*(\alpha)= n_1 \alpha+ n_2\alpha=(n_1 + n_2)\cdot\alpha,
\end{align*}
and similarly
\begin{align*}
\mu(f,g)^*(\beta) &= (f,g)^*\mu^*(\beta) = (f,g)^*(\beta\otimes 1 + 1\otimes\beta)\\
&=f^*(\beta) +g^*(\beta)= m_1 \beta+ m_2\beta=(m_1 + m_2)\cdot \beta.
\end{align*}
This confirms the additivity formula for the multi-degree.
In the proof of Lemma \ref{l4},  
we show that for any continuous self-map $f$ with multi-degree $\text{mdeg}(f)=(n,m)$, the integers $n,m$ must have
the same parity. We now proceed to show that this is the only constraint on the multi-degree of a map, which in particular, implies statement (d) in our
Theorem \ref{t:s3bundlesovers5}.

In the second paragraph of the proof of Lemma \ref{l4} 
we argued that every odd integer $r$ is realized as the degree of a self-map of $SU(3)$. 
Such a self-map $\tilde{f}_r\colon SU(3)\to SU(3)$ of odd degree $r$ was obtained by pulling back the bundle $SU(3)\stackrel{p}\longrightarrow S^5$ along a self-map 
$f_r\colon S^5\to S^5$ of odd degree $r$. Note in particular that $\text{mdeg}(\tilde f_r) = (1,r)$. Of course, the identity map $\mathrm{id}_{SU(3)}$ satisfies
$\text{mdeg}(\mathrm{id}_{SU(3)}) = (1,1)$, which by additivity of multi-degree tells us that the power map $\nu_s: A\mapsto A^{s}$ satisfies 
$\text{mdeg}(\nu_s) = (s,s)$, where $s$ is any integer. Finally, given an arbitrary
pair of integers $(n,m)$ of the same parity, we have that $r:=m-n+1$ is an {\it odd} integer, and hence $\tilde f_{m-n+1}$ is a self-map of multi-degree $(1, m-n+1)$.
This implies that the map $F_{nm}:SU(3)\rightarrow SU(3)$ defined via $A\mapsto A^{n-1}\cdot \tilde f_{m-n+1}(A)$ will have multi-degree 
\begin{align*}
\text{mdeg}(F_{nm}) &= \text{mdeg}(\nu_{n-1})+ \text{mdeg}(\tilde f_{m-n+1})\\
&=(n-1, n-1)+(1, m-n+1)=(n,m),
\end{align*}
as desired. This completes the proof of statement (d) in our Theorem \ref{t:s3bundlesovers5}.
\end{proof}

\begin{rem}
For maps from $SU(3)$ to $S^3\times S^5$, Wang \cite{Wang} is able to show that $\deg(f)$ cannot be congruent to $2$ mod $4$, correcting statement (c) in our
original paper. By pre-composing with suitable self-maps of $SU(3)$, it is sufficient to show that $(1,2)$ and $(2,1)$ cannot be realized as multi-degrees
of maps from $SU(3)$ to $S^3\times S^5$. We can easily rule out maps of multi-degree $(1,2)$, as composing such a map with the even degree maps 
$f:S^3\times S^5\rightarrow SU(3)$ produced in Lemma \ref{l1} would yield a self-map of $SU(3)$ of multi-degree $(1, 2k)$, contradicting the first paragraph
in our proof of Lemma \ref{l4}. On the other hand, we do not know how to rule out maps of multi-degree $(2,1)$. In \cite{Wang}, Wang obtains an obstruction by working in the homotopy group
$\pi_8(S^4)$. It is unclear whether one can obstruct such maps by only using cohomology operations.
\end{rem}

\subsection*{Acknowledgments}
We are grateful to Jim Fowler for suggesting the approach presented in the erratum. J.- F. Lafont was partially supported by the U.S.A. NSF, under grant DMS-1812028. C. Neofytidis was partially supported by the Swiss NSF, under grant FNS200021$\_$169685.

\end{document}